\documentclass[12pt]{amsart}
\usepackage{tikz}
\usepackage{caption}
\usepackage{subcaption}

\usepackage{graphicx}
\usepackage{color}
\usepackage{amsmath}
\usepackage{amssymb}
\newcommand{\beq}{\begin{eqnarray*}}
\newcommand{\feq}{\end{eqnarray*}}
\newcommand{\beqn}{\begin{eqnarray}}
\newcommand{\feqn}{\end{eqnarray}}

\newtheorem{theorem}{Theorem}[section]
\newtheorem{lemma}[theorem]{Lemma}

\theoremstyle{definition}

\theoremstyle{remark}

\numberwithin{equation}{section}

\allowdisplaybreaks

\begin{document}
\title[Global solutions for Euler-Poisson system]{Global solutions for the two dimensional Euler-Poisson system with attractive forcing}
\author{Yongki Lee}
\address{Department of Mathematical Sciences, Georgia Southern University, Statesboro,  30458}
\email{yongkilee@georgiasouthern.edu}
\keywords{Critical thresholds, Euler-Poisson equations}
\subjclass{Primary, 35Q35; Secondary, 35B30}
\begin{abstract} 
The Euler-Poisson(EP) system describes the dynamic behavior of many important physical flows. 
In this work, a Riccati system that governs the flow's gradient is studied. The evolution of divergence 
is governed by the Riccati type equation with several nonlinear/nonlocal terms. Among these, the vorticity accelerates divergence while others suppress divergence and enhance the finite time blow-up of a flow. The growth of the latter terms are related to the Riesz transform of density and non-locality of these terms make it difficult to study global solutions of the multi-dimensional EP system.  Despite of these, we show that the Riccati system can afford to have global solutions 
under a suitable condition, and admits global smooth solutions for a large set of initial configurations. 
To show this, we construct an auxiliary system in 3D space and find an invariant space of the system, then comparison with the original 2D system is performed.
   The present work generalizes several previous so-called restricted/modified EP models.
\end{abstract}
\maketitle


\section{Introduction}

We are concerned with the threshold phenomenon in two-dimensional Euler-Poisson (EP) equations. The pressureless Euler-Poisson equations in multi-dimensions are
\begin{subequations} \label{101}
    \begin{equation} \label{101a}
        \rho_t + \nabla \cdot (\rho \mathbf{u})=0,
    \end{equation}
    \begin{equation} \label{101b}
        \mathbf{u}_t + \mathbf{u} \cdot \nabla \mathbf{u} = k \nabla \Delta^{-1} (\rho-c_b),
    \end{equation}
\end{subequations}
which are the usual statements of the conservation of mass and Newton's second law. Here $k$ is a physical constant which parameterizes the repulsive $k>0$ or attractive $k<0$ forcing, governed by the Poisson potential $\Delta^{-1}(\rho -c_b)$ with constant $c_b>0$ which denotes background state.
  The local density $\rho=\rho(t,x)$ : $\mathbb{R}^+ \times \mathbb{R}^n \mapsto \mathbb{R^+} $  and the velocity field $\mathbf{u}(t,x)$ : $\mathbb{R}^+ \times \mathbb{R}^n \mapsto \mathbb{R}^{n}$  are the unknowns. This hyperbolic system \eqref{101} with non-local forcing describes the dynamic behavior of many important physical flows, including  plasma with collision, cosmological waves, charge transport,  and the collapse of stars due to self gravitation.

There is a considerable amount of literature available on the solution behavior of Euler-Poisson system. Global existence due to damping relaxation and with non-zero back-ground can be found in \cite{Wang01}. For the model without damping relaxation, construction of a global smooth irrotational solution in three dimensional space can be found in \cite{Guo98}. Some related results on two dimensional case can be found in \cite{JLZ14, LW14, IP13}. One the other hand, we refer to \cite{E96, WW06, MP90} for singularity formation and nonexistence results.

We focus our attention on the questions of global regularity versus finite-time blow-up  of solutions for \eqref{101}. Many of the results mentioned above leave open the question of global regularity of solutions to \eqref{101} subject to more general conditions on initial configurations, which are not necessarily confined to a ``sufficiently small" ball of any preferred norm of initial data. In this regard, we are concerned here with so called Critical Threshold (CT) notion, originated and developed in a series of papers by Engelberg, Liu and Tadmor  \cite{ELT01, LT02, LT03} and more recently in various models \cite{TT14, BL20, YT19}. The critical threshold in \cite{ELT01} describes the conditional stability of the one-dimensional Euler-Poisson system, where the answer to the question of global vs. local existence depends on whether the initial data crosses a critical threshold. Following \cite{ELT01}, critical thresholds have been identified for several \emph{one-dimensional} models, e.g., $2 \times 2$ quasi-linear hyperbolic relaxation systems \cite{LL09}, Euler equations with non-local interaction and alignment forces \cite{TT14, CCTT16}, traffic flow models \cite{YT19}, and damped Euler-Poisson systems \cite{BL20}.

Moving to the \emph{multi-dimensional setup}, the main difficulty lies with the non-local nature of the forcing term $\nabla \Delta^{-1} \rho$, and this feature was the main motivation for studying the so called ``restricted" or ``modified" EP models \cite{Y16, Y17, LT03, Tan14}, where the nonlocal forcing term is replaced by a local or semi-local one. The regularity of the (original) Euler-Poisson equations in $n>1$ dimensions remains an outstanding open problem.

The goal of this paper is showing that, under a suitable condition, two-dimensional Euler-Poisson system with attractive forcing can afford to have global smooth solutions for a large set of initial configurations. In section \ref{section 2}, we seek the evolution of $\nabla \mathbf{u}$ and derive a closed ordinary differential equations (ODE) system which is nonlinear and nonlocal, and relate/review many previous works with the derived ODE system. In section 3, we discuss the motivation and highlights of the present work. In addition to this, we state our main results about global solutions to the EP system. The details of the proofs of those main results are carried out in Sections 4 and 5.

$$$$

\section{Problem formulation and related works}\label{section 2}
In this work, we consider two-dimensional Euler-Poisson equations with attractive forcing \eqref{101}.
We are mainly concerned with a Riccati system that governs  $\nabla \mathbf{u}$.
In order to trace the evolution of $\mathcal{M}:=\nabla \mathbf{u}$, we differentiate \eqref{101b}, obtaining
\begin{equation}\label{Meqn_pre}
\partial_t \mathcal{M} + \mathbf{u} \cdot \nabla \mathcal{M} + \mathcal{M}^2 = k \nabla \otimes \nabla \Delta^{-1} (\rho -c_b) = k R[\rho-c_b],
\end{equation}
where $R[\cdot] $ is the $2 \times 2$ Riesz matrix operator, defined as
$$R[h]:=\nabla \otimes \nabla \Delta^{-1}[h]=\mathcal{F}^{-1}\bigg{\{} \frac{\xi_i \xi_j}{|\xi|^2} \hat{h}(\xi)  \bigg{\}}_{i,j=1,2} .$$

We let $\frac{D}{Dt}[\cdot] = [\cdot]'$ be the usual material derivative, $\frac{\partial}{\partial t} + u \cdot \nabla$.
We are concerned with the initial value problem (1.2) or
\begin{equation}\label{Meqn}
\frac{D}{Dt}\mathcal{M} + \left(
                   \begin{array}{cc}
                     \mathcal{M}_{11}^2 + \mathcal{M}_{12}\mathcal{M}_{21} & (\mathcal{M}_{11}+\mathcal{M}_{22})\cdot \mathcal{M}_{12} \\
                      (\mathcal{M}_{11}+\mathcal{M}_{22}) \cdot \mathcal{M}_{21} & \mathcal{M}_{12}\mathcal{M}_{21} +\mathcal{ M}_{22}^2 \\
                   \end{array}
                 \right)
 = k\left(
                         \begin{array}{cc}
                           R_{11} [\rho-c_b]& R_{12} [\rho-c_b]\\
                           R_{21}[\rho -c_b] & R_{22} [\rho-c_b]\\
                         \end{array}
                       \right).
 \end{equation}
 subject to initial data
$$(\mathcal{M}, \rho)(0, \cdot) = (M_0 , \rho_0).$$
The global nature of the Riesz matrix $R[ \cdot ]$, makes the issue of regularity for Euler-Poisson equations such an intricate question to solve.

We introduce several quantities with which we characterize the behavior of the velocity gradient tensor $\mathcal{M}$. These are the trace $d:=\mathrm{tr} \mathcal{M} = \nabla \cdot \mathbf{u}$, the vorticity $\omega : = \nabla \times \mathbf{u} = \mathcal{M}_{21} - \mathcal{M}_{12}$ and quantities $\eta:= \mathcal{M}_{11} - \mathcal{M}_{22}$ and $\xi := \mathcal{M}_{12} + \mathcal{M}_{21}$. 
Taking the trace of \eqref{Meqn}, one obtain
\begin{equation}\label{Deqn}
\begin{split}
 d'&= - (\mathcal{M}^2 _{11} + \mathcal{M}^2 _{22}) -2\mathcal{M}_{12}\mathcal{M}_{21} + k (R_{11} [\rho-c_b]  + R_{22} [\rho-c_b])\\
&=-\bigg{\{ } \frac{(\mathcal{M}_{11} + \mathcal{M}_{22})^2}{2} +  \frac{(\mathcal{M}_{11} - \mathcal{M}_{22})^2}{2}   \bigg{\}} +  \frac{(\mathcal{M}_{21} - \mathcal{M}_{12})^2}{2} -  \frac{(\mathcal{M}_{12} + \mathcal{M}_{21})^2}{2}   + k(\rho-c_b)\\
&=-\frac{1}{2}d^2 -\frac{1}{2}\eta^2 + \frac{1}{2}\omega^2 -\frac{1}{2}\xi^2 + k(\rho-c_b).
\end{split}
\end{equation}

We can see that the equation \eqref{Deqn} is a Ricatti-type equation.
One can view the evolution of $d$ as the result of a contest between negative and positive terms in \eqref{Deqn}.  Indeed,  the vorticity accelerates divergence while $\eta$ and $\xi$ suppress divergence and enhance the finite time blow-up of a flow. The growth of $\eta$ and $\xi$ are related to the Riesz transform of density and non-locality of these terms make it difficult to study global solutions of the multidimensional EP system.

Our approach in this paper is to study the evolutions of $d=\nabla \cdot \mathbf{u}$ and it shall be carried out by tracing the dynamics of  $\eta$, $\omega$ and $\xi$. From matrix equation \eqref{Meqn}, and \eqref{101a}, we obtain
\begin{subequations}
    \begin{equation}\label{Eeqn}
  	 \eta' +\eta d =k(R_{11}[\rho-c_b] - R_{22}[\rho-c_b]),
    \end{equation}
    \begin{equation}\label{Veqn}
    	\omega' + \omega d = k(R_{21}[\rho-c_b]-R_{12}[\rho-c_b])=0,
    \end{equation}
    \begin{equation}\label{Xeqn}
     \xi' + \xi d = k (R_{12}[\rho-c_b]+R_{21}[\rho-c_b]),
    \end{equation}
     \begin{equation}\label{Reqn}
    \rho' + \rho d=0.
        \end{equation}
        \end{subequations}
Here, one can explicitly  calculate $R[\cdot]$, (see \cite{Y16} for detailed calculations) i.e.,
\begin{equation}
(R_{i j}  [h])(x)  = p.v.\int_{\mathbb{R}^2} \frac{\partial^2}{\partial y_j \partial y_i}G(y)h(x - y) \, d y +  \frac{h(x)}{2\pi}\int_{|z| =1}   z_i z_j  \, dz,
\end{equation}
where $G(y)$ is the Poisson kernel in two-dimensions, and is given by
$$G(y) = \frac{1}{2\pi} \log |y|.$$
Due to the singular nature of the integral, we are lack of $L^{\infty}$ estimate of the $R_{ij}[\cdot]$.
        
From \eqref{Veqn} and \eqref{Reqn}, we derive
\begin{equation}\label{ome_rho}
\frac{\omega}{\omega_0}=\frac{\rho}{\rho_0}.
\end{equation}
One can also rewrite $\eta$ and $\xi$ in terms of $\rho$, by explicitly solving \eqref{Eeqn} and \eqref{Xeqn} (see \cite{Y16} ), we obtain
\begin{equation}\label{ex_rho}
\eta(t) =  \bigg{(} \frac{\eta_0}{\rho_0}  + \int^t _0 \frac{f_1(\tau)}{\rho(\tau)} \, d\tau \bigg{)}\rho(t) \ and \ \xi(t) =   \bigg{(} \frac{\xi_0}{\rho_0}  + \int^t _0 \frac{f_2(\tau)}{\rho(\tau)} \, d\tau \bigg{)}\rho(t),
\end{equation}
where
\begin{equation}\label{f_1}
f_1(t):=k(R_{11}[\rho-c_b] - R_{22}[\rho-c_b]) = \frac{k}{\pi} p.v.\int_{\mathbb{R}^2} \frac{-y^2 _1 + y^2 _2}{(y^2 _1 + y^2 _2)^2} \rho(t, x(t) - y ) \, dy,
\end{equation}
and
\begin{equation}\label{f_2}
f_2(t):=k (R_{12}[\rho-c_b]+R_{21}[\rho-c_b]) = \frac{k}{\pi} p.v.\int_{\mathbb{R}^2} \frac{-2y_1 y_2}{(y^2 _1 + y^2 _2)^2} \rho(t, x(t) - y ) \, dy.
\end{equation}
Here, all functions of consideration are evaluated along the characteristic, that is, for example,
$f_i (t)=f_i(t, x(t))$ and $\eta(t)=\eta(t, x(t))$, etc.

Using \eqref{ome_rho} and \eqref{ex_rho} we can rewrite \eqref{Deqn} in a manner that all non-localities are absorbed in the coefficient of $\rho^2$. That is, together with \eqref{Reqn}, we obtain closed system
\begin{equation}\label{ode1_intro}
\left\{
  \begin{array}{ll}
    d' = -\frac{1}{2}d^2 + A(t)\rho^2 + k(\rho -c_b), \\
    \rho' = -\rho d,\\
  \end{array}
\right.
\end{equation}
where
\begin{equation}\label{A_eqn}
A(t):=\frac{1}{2} \bigg{[} \bigg{(} \frac{\omega_0}{\rho_0} \bigg{)}^2 - \bigg{(} \frac{n_0}{\rho_0} + \int^t _0 \frac{f_1(\tau)}{\rho(\tau)} \, d \tau \bigg{)}^2 - \bigg{(} \frac{\xi_0}{\rho_0} + \int^t _0 \frac{f_2(\tau)}{\rho(\tau)} \, d \tau \bigg{)}^2  \bigg{]}.
\end{equation}

In this work, we are concern with \eqref{ode1_intro}, subject to initial data
$$(\nabla \mathbf{u}, \rho)(0,\cdot)=(\nabla \mathbf{u}_0, \rho_0).$$

To put our study in a proper perspective we recall several recent works in the form of \eqref{ode1_intro}. It turns out that many of so-called restricted/modified can be reinterpreted within the scope of \eqref{ode1_intro}.

$\bullet$ Chae and Tadmor \cite{CT08}  proved the finite time blow-up for solutions of $k<0$ case, assuming \emph{vanishing initial vorticity}. Indeed, setting $\omega_0 =0$ in \eqref{A_eqn}  gives $A(t) \leq 0$, and this allows to derive 
$$d' \leq -\frac{1}{2}d^2 +  k(\rho -c_b).$$
Using this ordinary differential inequality, upper-threshold for finite time blow-up of solution was identified. Later Cheng and Tadmor \cite{CT09} improved the result of \cite{CT08} using the delicate ODE phase plane argument.

$\bullet$  Liu and Tadmor \cite{LT02, LT03} introduced the restricted Euler-Poisson (REP) system which is obtained from \eqref{Meqn} by restricting attention to the local isotropic trace $\frac{k}{2} (\rho -c_b) I_{2 \times 2}$ of the global coupling term $kR[\rho- c_b]$. One can also obtain the REP by letting $f_i \equiv 0$, $i=1,2$ in \eqref{A_eqn}. That is,
$$d' = -\frac{1}{2}d^2 + \frac{\beta}{2}\rho^2 + k(\rho -c_b), \ \ \beta=\frac{\omega^2 _0 - \eta^2 _0 - \xi^2 _0}{\rho^2 _0}.$$
The dynamics of $(\rho, d)$ of this ``localized" EP system was studied, and it was shown that in the repulsive case, the REP system admits so called critical threshold phenomena.

$\bullet$ Slight generalization of the REP was introduced in \cite{Y17}. This ``weakly" restricted EP can also be obtained by letting $f_1\equiv 0$ \emph{only} in \eqref{A_eqn}. Indeed, $f_1\equiv 0$ implies $A(t) \leq \frac{1}{2}[(\frac{\omega_0}{\rho_0})^2 - (\frac{\eta_0}{\rho_0})^2]$ and
$$d' \leq -\frac{1}{2}d^2 + \frac{\alpha}{2}\rho^2 + k(\rho -c_b), \ \ \alpha=\frac{\omega^2 _0 - \eta^2 _0 }{\rho^2 _0}.$$
Threshold conditions for finite time blow-up were identified for attractive and repulsive cases.

$\bullet$ While the dynamics of $d$ in the above reviewed models are governed by \emph{local} quantities, the model in \cite{Y16} strives to maintain some \emph{global} nature of $A(t)$. That is, the author assumed that
$$\bigg{|} \int^t _0 \frac{f_i(\tau)}{\rho(\tau)} \, d \tau \bigg{|} \leq C \int^t _0 \frac{1}{\rho(\tau)} \, d \tau, \ \ i=1,2$$
for some constant $C$, and obtained upper-thresholds for finite time blow-up for attractive and repulsive cases.

$\bullet$  Tan \cite{Tan14} assumed that
$$\bigg{|}\frac{f_i (t)}{\rho(t)}  \bigg{|} \leq C, \ \ i=1,2$$
and proved a global existence of solution for  repulsive case using some scaling argument.

$$$$

\section{Highlights of the paper and main results }
We first address some motivation of this work.  The difficulty lies with the nonlocal/singular nature of the Riesz transform, which fails to map $L^{\infty}$ data to $L^{\infty}$. Thus, main obstacle in handling the dynamics of $d$ in \eqref{ode1_intro} is the lack of an accurate description for the propagations of $f_i (t,x(t))$ in \eqref{f_1} and \eqref{f_2}. This, in turn, makes difficult to answer the questions of global regularity versus finite-time breakdown of solutions for \eqref{ode1_intro}.

From \eqref{A_eqn}, we know the initial value, and the uniform \emph{upper bound} of $A(t)$. That is,

$$A(0)=\frac{1}{2} \bigg{[} \bigg{(} \frac{\omega_0}{\rho_0} \bigg{)}^2 - \bigg{(} \frac{n_0}{\rho_0}  \bigg{)}^2 - \bigg{(} \frac{\xi_0}{\rho_0}  \bigg{)}^2  \bigg{]}$$
and
\begin{equation}\label{A_upper}
A(t) \leq \frac{1}{2} \bigg{(} \frac{\omega_0}{\rho_0} \bigg{)}, \ for \ all \ t \geq 0,
\end{equation}
as long as $A(t)$ exists.

However, we do not know if there exists any lower bound of $A(t)$. It is possible that $A(t) \rightarrow -\infty$ in finite/infinite time or remains uniformly bounded below for all time. This is because, as mentioned earlier, there is no $L^\infty$ bound of $f_i$. For each fixed $t$, we know that  $f_i (t,\cdot) \in \mathrm{BMO}(\mathbb{R}^2)$ (bounded mean oscillation, see e.g. \cite{S93}), and this implies that
$$f_i (t,\cdot) \in L^p _{loc} (\mathbb{R}^2), \ \ 1\leq p < \infty.$$ 

Since $A(t)$ is bounded above, we are left with only two possible cases
 under non-vacuum condition $\rho_0 >0$ (thus $\rho(t) >0$ from the second equation of \eqref{ode1_intro}):
 
\textbf{Case I:} Finite time blow-up of $A(t)$. That is, $$\lim_{t \rightarrow t^* -} A(t) = -\infty,$$ where $t^* < \infty$. This corresponds to 
$$\lim_{t \rightarrow t^* -}\bigg{|} \int^t _0 f_i (\tau, x(\tau)) \, d\tau \bigg{|} = \infty,$$
and this is possible because for each $t$, $f_i (t, \cdot)$ need not be locally bounded, so  $f_i(t, x(t))$ can be unbounded along some part of the characteristic.  In this case, we can easily see that  $$\lim_{t \rightarrow t^* -} d(t) = -\infty$$ as well, and $\rho, d$ blow-up in finite time. Indeed, suppose not, then since $t^* <\infty$, 
$$\rho(t^*) = \rho_0 e^{-\int^{t^*} _0 d(\tau, x(\tau)) \, d \tau} >\epsilon,$$
for some $\epsilon>0$. Applying this and $\lim_{t \rightarrow t^* -} A(t) = -\infty$ to \eqref{ode1_intro}, we obtain $\lim_{t \rightarrow t^* -} d' (t) = -\infty$ and this is contradiction.

\textbf{Case II:} $A(t)$ uniformly bounded below, or blows-up at infinity. That is, there exists some function  $h(t)$ such that
$A(t) \geq h(t)$, for all $t\geq 0$ and 
$$h(t) \rightarrow -\infty \  as \  t \rightarrow \infty.$$ 

The main contribution of this work is investigating \eqref{ode1_intro} under the condition in Case II, and we show that the Riccati structure  can afford to have global solutions even if $A(t) \rightarrow -\infty$, depending on $A(t)$'s rate of decreasing. More precisely, we show that the nonlinear-nonlocal system \eqref{ode1_intro} admits global smooth solutions for a large set of initial configurations provided that
\begin{equation}\label{exp_condi_1}
A(t) \geq -\alpha_1 e^{\beta_1 t}, \ for \ all \ t,
\end{equation}
where $\alpha_1$ and $\beta_1$ are some positive constants.

We can also show a similar result under the condition that
\begin{equation}\label{poly_condi_1}
A(t) \geq -(\alpha_2 t + \beta_2)^s, \ for \ all \ t,
\end{equation}
where $\alpha_2, \beta_2 >0$ and $s\geq 1$.  Of course the condition in \eqref{exp_condi_1} may imply the one in \eqref{poly_condi_1} depending on $\alpha_i$ and $\beta_i$. But it is worth to observe the difference between two initial configurations that lead to the global existence of the system under  two different conditions \eqref{exp_condi_1} and  \eqref{poly_condi_1}. Furthermore, there is a slight difference in our proofs when handling  \eqref{exp_condi_1} and \eqref{poly_condi_1}. So we consider both cases for the completeness.


From now on, in \eqref{ode1_intro}, we assume that $k=-1$ and $c_b=1$, because these constants are not essential in our analysis. Also, we set $\alpha_i=\beta_i=1$, $i=1,2$. We shall consider 
\begin{equation}\label{ode2_intro}
\left\{
  \begin{array}{ll}
    d' = -\frac{1}{2}d^2 + A(t)\rho^2 - (\rho -1), \\
    \rho' = -\rho d,\\
  \end{array}
\right.
\end{equation}
subject to initial data $(\nabla \mathbf{u}, \rho)(0,\cdot)=(\nabla \mathbf{u}_0, \rho_0),$
where
\begin{equation}\label{A_eqn_2}
A(t):=\frac{1}{2} \bigg{[} \bigg{(} \frac{\omega_0}{\rho_0} \bigg{)}^2 - \bigg{(} \frac{n_0}{\rho_0} + \int^t _0 \frac{f_1(\tau)}{\rho(\tau)} \, d \tau \bigg{)}^2 - \bigg{(} \frac{\xi_0}{\rho_0} + \int^t _0 \frac{f_2(\tau)}{\rho(\tau)} \, d \tau \bigg{)}^2  \bigg{]},
\end{equation}
with
\begin{equation}\label{poly_condi_2}
A(t) \geq -( t + 1)^s, \ for \ all \ t, 
\end{equation}
or
\begin{equation}\label{exp_condi_2}
A(t) \geq - e^{ t}, \ for \ all \ t.
\end{equation}

Here, we note that either \eqref{poly_condi_2} or  \eqref{exp_condi_2} already assumes that $A(0)\geq-1$, that is,
$$\frac{1}{2} \bigg{[} \bigg{(} \frac{\omega_0}{\rho_0} \bigg{)}^2 - \bigg{(} \frac{n_0}{\rho_0}  \bigg{)}^2 - \bigg{(} \frac{\xi_0}{\rho_0}  \bigg{)}^2  \bigg{]} \geq-1.$$
But this does not restrict our result, since one can always find $\alpha_i$ and $\beta_i$ that satisfy \eqref{exp_condi_1} and  \eqref{poly_condi_1} for any $A(0)$.

To present our results, we write \eqref{Meqn_pre} or \eqref{Meqn} here again, with the second equation of \eqref{ode2_intro}, to establish the two-dimensional Euler-Poisson system:
\begin{equation}\label{EP_system}
\left\{
  \begin{array}{ll}
    \mathcal{M}' + \mathcal{M}^2 =-R[\rho -1], \\
    \rho' = -\rho\cdot \mathrm{tr}(\mathcal{M}),\\
  \end{array}
\right.
\end{equation}
subject to initial data $(\mathcal{M},\rho)(0,\cdot)=(\mathcal{M}_0 , \rho_0)$.
We note that the global regularity follows from the standard boot-strap argument, once an \emph{a priori} estimate on $\| \mathcal{M}(\cdot, \cdot) \|_{L^{\infty}}$ is obtained. Also, under \eqref{poly_condi_2} or \eqref{exp_condi_2}, $\mathcal{M}=\nabla \mathbf{u}$ is completely controlled by $d=\nabla \cdot \mathbf{u}$ and $\rho$:
$$\|\mathcal{M}(t,\cdot) \|_{L^{\infty}[0,T]} \leq C_T \cdot \| (\mathrm{tr}\mathcal{M} , \rho) \|_{L^{\infty} [0,T]}.$$

Our goal of this work is to prove the following results.

\begin{theorem}\label{thm_poly}
Consider the Euler-Poisson system, \eqref{EP_system}  with \eqref{poly_condi_2}. 
If $$(\rho_0 , d_0) \in \Omega_{ s}:= \{(\rho ,d) \in \mathbb{R}^2  |  \rho>0 , d >0, \ and \  d > m^* \rho - n^*      \},$$
 then the solution of the Euler-Poisson system remains smooth for all time.
 Here, $m^* = m_1 m^M _2$ and $n^* = n_1 n^N _2$ are constants satisfying
\begin{equation}\label{condi_parta}
  0< N < 1, \ M>N+s, \ n_1 > 1+\sqrt{3}, \ n_2 >1, \ m_1 > \sqrt{2},
  \end{equation}
and
$$m_2 > \max \bigg{[}n_2, \bigg{(} \frac{m_1 n_1 + M m^2 _1 n_1 + n^2 _1}{m^2 _1} \bigg{)}^{1/(1-N)},    \bigg{(} \frac{n_1}{2m_1}(1+\sqrt{5}) \bigg{)}^{1/(M-N-s)} \bigg{]}.$$ 
\end{theorem}

\begin{theorem}\label{thm_exp}
Consider the Euler-Poisson system,  \eqref{EP_system}  with \eqref{exp_condi_2}. 
If $$(\rho_0 , d_0) \in \Omega:= \{(\rho ,d) \in \mathbb{R}^2  |  \rho>0 , d >0, \ and \  d > m^* \rho - n^*      \},$$
 then the solution of the Euler-Poisson system remains smooth for all time.
 Here, $m^* = m_1 m^M _2$ and $n^* = n_1 n^N _2$ are constants satisfying
\begin{equation}\label{condi_partb}
  N < 0, \ M>m_2 \sqrt{2}, \ m_1 > \sqrt{2}, \ m_2 > n_2, \ n_1 > 0,
  \end{equation}
and
$$n_2 > \max \bigg{[} 1, \bigg{(}  \frac{n_1}{2m_1}(1+\sqrt{5})  \bigg{)}^{1/(M-N-1)} \bigg{]}.$$ 
\end{theorem}

\textbf{Remarks.} Some remarks are in order at this point.

(1) The inequalities in theorems are explicit, so one can easily see that $\Omega_s$ and $\Omega$ are non-empty sets. Also, in Theorem \ref{thm_exp}, we note that $M-N-1$ is positive. This is because $m_2 > n_2 >1$ so that $M>\sqrt{2}$.

(2) For simplicity, we set $\alpha_i = \beta_i =1$ in \eqref{exp_condi_1} and \eqref{poly_condi_1} for the theorems. Our method work equally well for other general constants $\alpha_i$ and $\beta_i$.

(3) We note that main condition in our theorems $d_0 > m^* \rho_0 - n^*$ resembles the one in \emph{one-dimensional} EP system   \cite{ELT01}. Indeed, the critical threshold in 1D EP system depends on the relative size of the initial velocity gradient and initial density. More precisely, it was shown that the system admits a global solution if and only if
$$u_{0x}(x) \geq \rho_0 (x)-1.$$

(4) As discussed in Case I,  \emph{finite time} blow-up of $|A(t)|$ leads to the blow-ups of $\rho$ and $d$ in finite time. $A(t)$, the coefficient of $\rho^2$,  has the uniform upper bounds in \eqref{A_upper}.  Thus, the main contribution of the theorems is that the Riccati structure \eqref{EP_system} affords to have  solutions  while $A(t)$ freely moves under conditions \eqref{exp_condi_1} or \eqref{poly_condi_1}.
In particular, this includes that the system admits global solutions even though $A(t)$ blows up at infinity, as long as the blow-up rate is not severe.

\begin{figure}[ht]\begin{center}
\includegraphics[width=70mm]{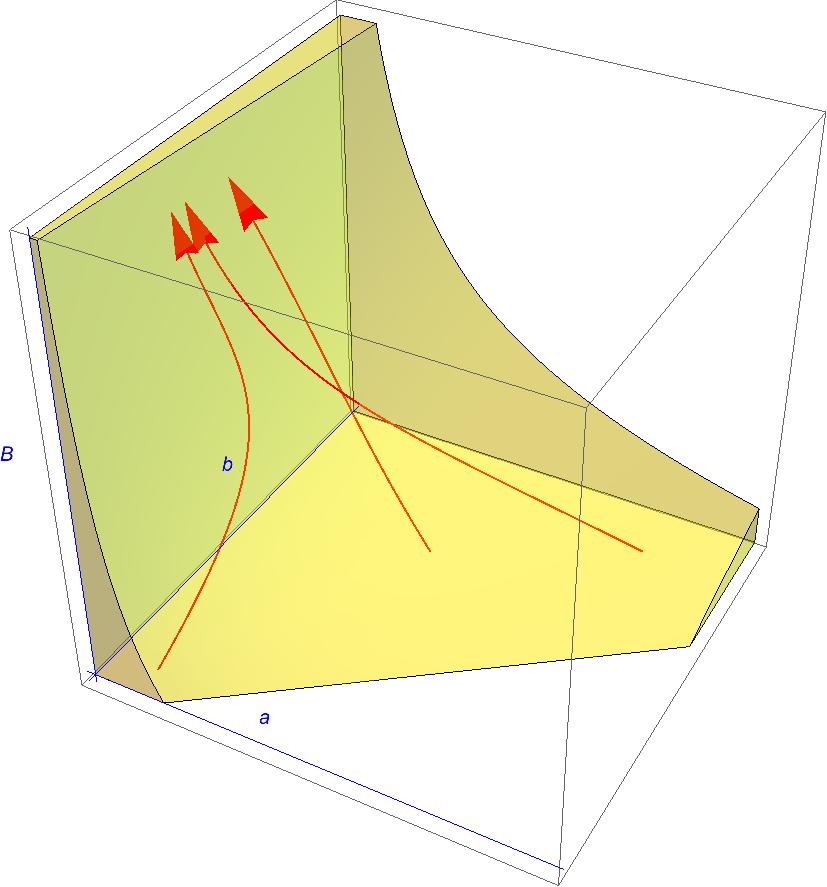}
\caption{3D  invariant space of the auxiliary system}\label{invariant}
\end{center}
\end{figure}

The next sections are devoted to the proofs of the theorems. In order to prove the theorems, we introduce an $3 \times 3$ auxiliary system
\begin{equation*}
\left\{
  \begin{array}{ll}
    b' = \tilde{H}(b,a,B), \\
    a' = -ba,\\
    B'=G(B),
  \end{array}
\right.
\end{equation*}
and find a three-dimensional invariant space of the system, where all trajectories if they start from inside this space will stay encompassed at all time, see Figure \ref{invariant}. Then, we compare the auxiliary system with \eqref{ode1_intro},
\begin{equation*}
\left\{
  \begin{array}{ll}
    d' = H(d,\rho,A(t)), \\
    \rho' = -\rho d.\\
  \end{array}
\right.
\end{equation*}
The key parts of the proofs are constructing the surface that determines the three-dimensional invariant space of the auxiliary system, and establishing monotonicity between the auxiliary system and the original system.




\section{Proof of Theorem \ref{thm_poly}}\label{section_poly}

We start this section by considering the following nonlinear ODE system with the time dependent coefficient,
\begin{equation}\label{ode1_poly}
\left\{
  \begin{array}{ll}
    \dot{b} = -b^2 /2 -(t+1)^s a^2 -a +1, \\
    \dot{a} = -ba.\\
  \end{array}
\right.
\end{equation}

Setting $B(t)=t+1$, one can rewrite the system as follows:
\begin{equation}\label{3by3_system}
\left\{
  \begin{array}{ll}
    \dot{b} = -\frac{1}{2}b^2 - B^s a^2 -a +1, \\
    \dot{a} = -ba\\
    \dot{B} =1
  \end{array}
\right.
\end{equation}
with $(a , b , B)\big{|}_{t=0} = (a_0 , b_0 , B_0 =1).$

We shall find set of initial data for which the solution of \eqref{3by3_system} exists for all time. Consider surface
$$b=m(B-1) a - n(B-1), \ B\geq1$$
in $(a,b, B)$ space where $m(\cdot)$ and $n(\cdot)$ are positive on $[0,\infty$) and continuously differentiable.
We find conditions on $m(\cdot)$ and $n(\cdot)$ such that trajectory $(a , b, B)$ stays on one side of 
\begin{equation}\label{f_surface}
F(a , b , B):=b - m(B-1) a + n(B-1)=0.
\end{equation}
In order to do that, it requries
\begin{equation}\label{dot_product}
\langle \dot{a}, \dot{b}, \dot{B} \rangle \cdot \nabla F >0,
\end{equation}
on the surface $F(a ,b , B)=0$, where
$$\nabla F = \langle -m(B-1), 1, -m'(B-1)a + n'(B-1) \rangle.$$ 

Upon expanding \eqref{dot_product} and substituting \eqref{f_surface}, the left hand side of \eqref{dot_product} can be written as
\begin{equation*}
\begin{split}
&\big{\langle} -ba,  -\frac{1}{2}b^2 - B^s a^2 -a +1 , 1 \big{\rangle} \cdot \langle -m(B-1), 1, -m'(B-1)a + n'(B-1) \rangle \\
&\Rightarrow ba m(B-1) -\frac{1}{2}b^2 - B^s a^2 -a +1 -m'(B-1)a + n'(B-1)\\
&\Rightarrow (ma -n) a m - \frac{1}{2}(ma -n)^2 - B^s a^2 - a +1 -m' a + n'\\
&\Rightarrow \bigg{(} \frac{1}{2}m^2 - B^s \bigg{)} a^2 - (1+m')a - \frac{1}{2}n^2 + n' +1. 
\end{split}
\end{equation*}
Here and below $m$ and $n$ are evaluated at $B-1$. Thus, on the surface $F(a ,b , B)=0$, \eqref{dot_product} is equivalent to
\begin{equation}\label{qd_1}
\bigg{(} \frac{1}{2}m^2 - B^s \bigg{)} a^2 - (1+m')a - \frac{1}{2}n^2 + n' +1 >0.
\end{equation}

We will find $m$ and $n$ such that the above inequality holds for some set of $(a, b ,B)$. The inequality is quadratic in $a$, and the nonnegative root of the quadratic equation is given by
$$R_* (B-1) : = \frac{(1+m') + \sqrt{(1+m')^2 + (m^2 - 2B^s)(n^2 -2n' -2)}}{m^2 - 2B^s},$$
provided that 
\begin{equation}\label{leading_disc}
m^2 - 2B^s >0 \ \  and \ \ D=(1+m')^2 + (m^2 - 2B^s)(n^2 -2n' -2) \geq 0.
\end{equation}
For each fixed $B$, $\frac{n}{m}$ is the $a$-intercept of  $F(a ,b , B)=0$, so 
\begin{equation}\label{root_ineq}
R_* (B-1)< \frac{n(B-1)}{m(B-1)}
\end{equation}
implies
\eqref{qd_1}, for all $a \geq \frac{n(B-1)}{m(B-1)}$. 

Expanding, completing square and simplifying \eqref{root_ineq} give
\begin{equation*}
\begin{split}
& m \sqrt{(1+m')^2 + (m^2 - 2B^s)(n^2 -2n' -2)} < m^2 n -2B^s n - m - m m'\\
&\Rightarrow 2(m^2 - 2B^s) \big{\{}  (1+n')m^2 - (1+m')mn - B^s n^2 \big{\}} > 0.
\end{split}
\end{equation*}
Since $m^2 > 2B^s$, it suffices to have
\begin{equation*}
(1+n')m^2 - (1+m')mn - B^s n^2 >0.
\end{equation*}
Note that $m$ and $n$ are evaluated at $B-1$, so let $$x=B-1,$$ and writing the above inequality in terms of $x$ gives,
\begin{equation}\label{qd_2}
(1+n'(x))m^2 (x) - (1+m' (x))m(x)n(x) - (x+1)^s n^2(x) >0.
\end{equation}

\begin{figure}[h]  
\centering 
  \begin{subfigure}[b]{0.45\linewidth}
\begin{tikzpicture}
\draw[ultra thick,->] (0,0) -- (6,0) node[anchor=north west] {a};
\draw[ultra thick,->] (0,0) -- (0,4) node[anchor=south east] {b};
\draw [thick, red] (3,0) -- (5.5,3.5)  node[align=left, left]
{$b=ma - n$};;
\draw [fill] (1.5+0.4,.0)  node[below]{$R_*$} circle [radius=.07];
\draw [fill] (3,.0)  node[below]{$\frac{n}{m}$} circle [radius=.07];
\draw [fill] (4.5,.0)  node[below]{$a^* $} circle [radius=.07];
  \fill [gray, opacity=0.2, domain=0:2, variable=\x]
      (0, 0.01)
      -- plot (0, 3.5)
      -- plot (5.5, 3.5)
      -- (3, 0.01)
      -- cycle;
\draw[thick, blue, ->] (1.5/2,0) -- (1.5/2,0.5);
\draw[thick, blue, ->] (1.5 + 1.5/2,0) -- (1.5 + 1.5/2,0.5);
\draw[thick, blue, ->] (1.5 + 1.5/2,0) -- (1.5 + 1.5/2,0.5);
\draw[thick, blue, ->] (3 + 1.5/2,0) -- (3 + 1.5/2,0.5);
\draw[thick, blue, ->] (4.5 + 1.5/2,0) -- (4.5 + 1.5/2,-0.5);
\draw[thick, blue, ->] (3.8, 3.5*3.8/2.5 - 3.5*3/2.5) --  (3.8 -0.6, 3.5*3.8/2.5 - 3.5*3/2.5 -0.4);
\draw[thick, blue, ->] (4.8, 3.5*4.8/2.5 - 3.5*3/2.5) --  (4.8 -0.6, 3.5*4.8/2.5 - 3.5*3/2.5 -0.4);
\end{tikzpicture}
    \caption{$B=1$ initially} \label{fig:M1}  
  \end{subfigure}
\begin{subfigure}[b]{0.45\linewidth}
\begin{tikzpicture}
\draw[ultra thick,->] (0,0) -- (6,0) node[anchor=north west] {a};
\draw[ultra thick,->] (0,0) -- (0,4) node[anchor=south east] {b};
\draw [thick, red] (3-1.5,0) -- (5.5-1.5-1.5,3.5)  node[align=left, right]
{$b=ma - n$};;
\draw [fill] (1.5+0.4-1.5,.0)  node[below]{$R_* $} circle [radius=.07];
\draw [fill] (3-1.5,.0)  node[below]{$\frac{n}{m}$} circle [radius=.07];
\draw [fill] (4.5-1.5,.0)  node[below]{$a^* $} circle [radius=.07];
  \fill [gray, opacity=0.2, domain=0:2, variable=\x]
      (0, 0.01)
      -- plot (0, 3.5)
      -- plot (5.5-1.5-1.5, 3.5)
      -- (3-1.5, 0.01)
      -- cycle;
\draw[thick, blue, ->] (1.5/2,0) -- (1.5/2,0.5);
\draw[thick, blue, ->] (1.5 + 1.5/2,0) -- (1.5 + 1.5/2,0.5);
\draw[thick, blue, ->] (1.5 + 1.5/2,0) -- (1.5 + 1.5/2,0.5);
\draw[thick, blue, ->] (3 + 1.5/2,0) -- (3 + 1.5/2,-0.5);
\draw[thick, blue, ->] (4.5 + 1.5/2,0) -- (4.5 + 1.5/2,-0.5);
\draw[thick, blue, ->]   (1.75 , 3.5*1.75 - 3.5*1.5 ) --     (1.75-0.5 , 3.5*1.75 - 3.5*1.5 -0.5 );
\draw[thick, blue, ->]   (2.3 , 3.5*2.3 - 3.5*1.5 ) --  (2.3-0.5 , 3.5*2.3 - 3.5*1.5 -0.5);
\end{tikzpicture}
\caption{As $B$ increases} \label{fig:M2}  
\end{subfigure}
\caption{Cross-section of the ``shrinking" invariant space}\label{fig1}
\end{figure}
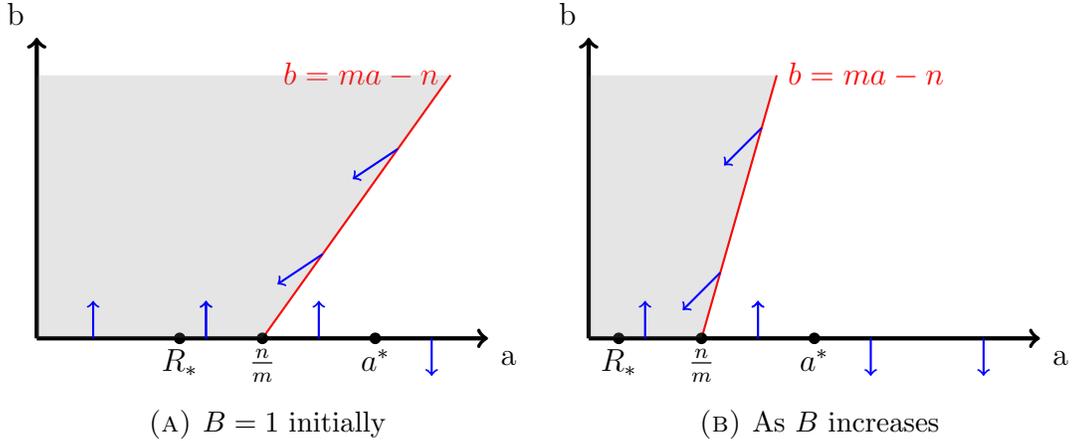

$$$$

\textbf{Construction of} $\mathbf{n(x)}$ \textbf{and} $\mathbf{m(x)}$. We prove the existences of $m(x)$ and $n(x)$. More precisely, we find simple polynomials 
$$n(x):=n_1 (x+n_2)^N, \ \  and \ \ m(x)=m_1 (x+m_2)^M$$
that satisfy \eqref{leading_disc} and  \eqref{qd_2} for all $x\geq 0$.  We want to emphasize the method and not the technicalities, so our construction here may not optimal, and one may obtain sharper functions $n(x)$ and $m(x)$ later.

First, we consider
$$m^2 (B-1) -2B^s = m^2 _1 (x+m_2)^{2M} -2(x+1)^s$$
in  \eqref{leading_disc}. From \eqref{condi_parta}, since $m_2 > n_2 >1$ and $M > N+s > s$,
$$m_1 > \sqrt{2}$$
implies positivity of $m^2 (B-1) -2B^s$ for all $B\geq 1$.

Next, we prove positivity of 
$$n^2 - 2n' -2 = n^2 _1 (x+n_2)^{2N} -2n_1 N (x+n_2)^{N-1} -2,$$
which in turn implies $D\geq 0$ in   \eqref{leading_disc}.
Since $0<N<1$, $x\geq 0$ and $n_2 >1$, we have $(x+n_2)^{2N} >1$ and $(x+n_2)^{N-1} < 1$. Thus, it suffices to show
$$n^2 _1 -2n_1 N -2 >0.$$
The left hand side is quadratic in $n_1$, so the inequality holds if
$$n_1 > N+\sqrt{N^2 +2}.$$
Since $N<1$,  $$n_1 > 1+ \sqrt{3}$$ gives the desired result.


Now, the following series of lemmata is useful to construct the ``shrinking'' invariant space in Figure \ref{fig1}.

\begin{lemma}\label{lemma_onsurface} 
Assume $1<n_2 < m_2$.
If $M>\frac{N+1 +s}{2}$, $0<N<1$ and 
\begin{equation}\label{lem_condi}
m^{1-N} _2> \frac{m_1 n_1 + M m^2 _1 n_1 + n^2 _1}{m^2 _1},
\end{equation}
then 
\begin{equation*}
(1+n'(x))m^2 (x) - (1+m' (x))m(x)n(x) - (x+1)^s n^2(x) >0,
\end{equation*}
for all $x\geq 0$. Thus \eqref{dot_product} holds.
\end{lemma}

\begin{proof}
We show that
\begin{equation}
\underbrace{(1+n')m^2}_{LHS} > \underbrace{(1+m')mn + n^2 (x+1)^s}_{RHS}.
\end{equation}
Note that
$$LHS > m^2 = m^2 _1 (x+m_2)^{2M}.$$
Also,
\begin{equation*}
\begin{split}
RHS&=mn+m' mn + n^2 (x+1)^s\\
&=m_1 n_1 (x+n_2)^N (x+m_2)^M + Mm^2 _1 n_1  (x+m_2)^{2M-1}  (x+n_2)^N +n^2 _1 (x+n_2)^{2N} (x+1)^s\\
&< m_1 n_1 (x+m_2)^{M+N} + Mm^2 _1 n_1 (x+m_2)^{2M+N-1} + n^2 _1 (x+m_2)^{2N+s}\\
&<(m_1 n_1 + M m^2 _1 n_1 + n^2 _1)(x+m_2)^{2M+N-1}.
\end{split}
\end{equation*}
Here, the assumptions $1<n_2 < m_2$ and $M>\frac{N+1 +s}{2}$ (therefore $M>1$) are used to derive inequalities. 

Now, it suffice to prove
$$m^2 _1 (x+m_2)^{2M} >(m_1 n_1 + M m^2 _1 n_1 + n^2 _1)(x+m_2)^{2M+N-1},$$
which is equivalent to
$$\frac{m^2 _1}{m_1 n_1 + M m^2 _1 n_1 + n^2 _1} > \frac{1}{(x+m_2)^{1-N}}.$$
We see that the right hand side of the above inequality is decreasing in $x \in [0, \infty)$, so \eqref{lem_condi} gives the desired result.
\end{proof}

From Lemma \ref{lemma_onsurface}, we can see that any trajectory on $F(a,b,B)=0$ can't cross $F(a,b,B)=0$ from left to right. We continue to construct an invariant region for the system \eqref{3by3_system}. It is easy to see that trajectory  $(a, 0, B)$ on the $aB$ plane moves upward if  $a< a^* $ (downward if $a > a^*$) where
$$a^* (B-1)=\frac{1-\sqrt{1+4B^s}}{-2B^s}.$$
In order to secure the invariant region we need 
$$a^* (B-1) > \frac{n(B-1)}{m(B-1)},$$
for all $B \geq 1$. See Figure \ref{fig1}. This is fulfilled by the following lemma.

\begin{lemma}\label{lemma_ontheaxis}
If $M>N+s$ and, 
$$m_2 >  \bigg{\{} \frac{n_1}{2m_1}(1+\sqrt{5}) \bigg{\}}^{1/(M-N-s)},$$
then it holds,
$$a^* (x) > \frac{n(x)}{m(x)},$$
for all $x\geq 0$.
\end{lemma}
\begin{proof}
Since $a^* (x)= \frac{-1 + \sqrt{1 + 4(x+1)^s}}{2(x+1)^s}$,  we show that
$$\underbrace{\frac{-1 + \sqrt{1 + 4(x+1)^s}}{2(x+1)^s}}_{LHS} > \underbrace{ \frac{n_1 (x+n_2)^N}{m_1 (x+m_2)^M}}_{RHS}.$$
Since $1 < n_2 < m_2$, we have
$$RHS= \frac{n_1 (x+n_2)^N}{m_1 (x+m_2)^M} < \frac{n_1 (x+m_2)^N}{m_1 (x+m_2)^M}.$$
Thus, it suffice to prove
$$LHS=  \frac{-1 + \sqrt{1 + 4(x+1)^s}}{2(x+1)^s} > \frac{n_1}{m_1 (x+ m_2)^{M-N}},$$
or equivalently
$$-1 + \sqrt{1+ 4(x+1)^s} > \frac{2n_1(x+1)^s}{m_1(x+m_2)^{M-N}}.$$

Since $m_2 > 1$, the above inequality holds if
$$-1 + \sqrt{1+ 4(x+1)^s} >  \frac{2n_1 }{ m_1 (x+m _2)^{M-N-s}}. $$
Upon simplification, the inequality is reduced to
\begin{equation}
4(x+1)^s > \frac{4n^2 _1}{m^2 _1 (x+m_2)^{2(M-N-s)}} + \frac{4n_1}{m_1 (x+m_2)^{M-N-s}}.
\end{equation}
Since $M-N-s >0$, we note that the right hand side of the above inequality is decreasing in $x$ and the left hand side is increasing in $x$. Thus, it suffices to hold at $x=0$, that is,
$$4 > \frac{4n^2 _1}{m^2 _1}\cdot \frac{1}{ m^{2(M-N-s)} _2} + \frac{4n_1}{m_1}\cdot \frac{1}{m^{M-N-s} _2},$$
or equivalently 
$$m^2 _1 m^{2(M-N-s)} _2 - m_1 n_1 m^{M-N-s} _2 - n^2 _1 >0.$$
The left hand side is quadratic in $m^{M-N-s} _2$ and positive when $m^{M-N-s} _2$ is greater than the positive root of the quadratic equation. That is,
$$m^{M-N-s} _2 > \frac{m_1 n_1 + \sqrt{m^2 _1 n^2 _1 + 4 m^2 _1 n^2 _1}}{2m^2 _1}= \frac{n_1}{2m_1}(1+\sqrt{5}).$$
\end{proof}

We are done with constructing the ``shrinking"  invariant region, and it's property can be summarized as follows.
\begin{lemma}\label{lemma_invariant}
Consider \eqref{ode1_poly}. Let $\Omega:=\{(a ,b) \in \mathbb{R}^2  |  a>0 , b >0,  and \  b > m(0) a - n(0)      \}$.
If $(a_0 , b_0) \in \Omega$, then $0<b(t)$ and $a(t) \leq a_0$ for all $t\geq 0$.
\end{lemma}
\begin{proof}
Solving the second equation of \eqref{3by3_system} gives
\begin{equation}\label{a_eqn}
a(t)=a_0 e^{-\int^t _0 b(\tau) \, d \tau},
\end{equation}
and this implies that if $a_0 >0$, then $a(t)>0$, $t\geq 0$.  Together with this, Lemmata \ref{lemma_onsurface} and \ref{lemma_ontheaxis} give $b(t)>0$, $t\geq 0$.  Next, $a(t) \leq a_0$ is obtained from the positivity of $b(t)$ for all $t\geq 0$ and \eqref{a_eqn}.

\end{proof}

Now, the final step of the proof is to compare
\begin{equation}\label{comp_epsystem}
\left\{
  \begin{array}{ll}
    \dot{d} = -d^2/2 +A(t) \rho^2 -\rho +1, \\
    \dot{\rho} = -d\rho\\
  \end{array}
\right.
\end{equation}
with
\begin{equation}\label{comp_auxsystem}
\left\{
  \begin{array}{ll}
    \dot{b} = -b^2 /2 -(t+1)^s a^2 -a +1, \\
    \dot{a} = -ba.\\
  \end{array}
\right.
\end{equation}
We recall that
$$-(t+1)^s \leq A(t) \leq  \frac{1}{2}\bigg{(} \frac{\omega_0}{\rho_0} \bigg{)}^2, \ \ t\geq 0.$$
We show the monotonicity relation between two ode systems.
\begin{lemma}\label{lemma_comp}
\begin{equation*}
\left\{
  \begin{array}{ll}
	b(0) < d(0), \\
    0<\rho(0)<a(0)\\
  \end{array}
\right.
implies
\ 
\left\{
  \begin{array}{ll}
	b(t) < d(t), \\
    0<\rho(t)<a(t)\\
  \end{array}
\right.
\ for \ all \ t>0.
\end{equation*}
\end{lemma}

\begin{proof}
Suppose $t_1$ is the earliest time when the above assertion is violated. Consider
$$a(t_1) = a(0) e^{-\int^{t_1} _ 0 b(\tau) \, d \tau} > \rho(0) e^{-\int^{t_1} _0 d(\tau) \, d \tau} = \rho (t_1).$$
Therefore, it is left with only one possibility that $d(t_1) = b(t_1).$ Consider
\begin{equation}\label{diff_systems}
\dot{b}-\dot{d} =-\frac{1}{2} (b^2 -d^2) - (t+1)^s a^2 - A(t) \rho^2 - a + \rho.
\end{equation}
Since $b(t)-d(t) <0$ for $t<t_1$ and $b(t_1) - d(t_1)=0$, hence at $t=t_1$, we have
$$\dot{b}(t_1)-\dot{d} (t_1) \geq 0.$$
But the right hand side of \eqref{diff_systems},  when it is evaluated at $t=t_1$, is negative. Indeed
\begin{equation*}
\begin{split}
&-\frac{1}{2} (b^2 (t_1) -d^2 (t_1)) - (t_1 +1)^s a^2 (t_1) - A(t_1) \rho^2 (t_1) - a(t_1) + \rho(t_1)\\
&= - (t_1 +1)^s a^2 (t_1) - A(t_1) \rho^2 (t_1) - a(t_1) + \rho(t_1)\\
&=(t_1 +1)^s \big{(}  -a^2 (t_1) + \rho^2 (t_1) \big{)} +\rho^2 (t_1) \big{(} -(t_1+1)^s -A(t_1)  \big{)} - a(t_1) + \rho(t_1),
\end{split}
\end{equation*}
so $a(t_1) > \rho(t_1)$ and $-(t_1 +1)^s \leq A(t_1)$ give the desired result. This leads to the contradiction.
\end{proof}

\begin{lemma}\label{lemma_dbound}
 Consider \eqref{comp_epsystem}. If there exists $\rho_M  >0$ such that $\rho(t) \leq \rho_M$, $\forall t \geq 0$, then $d(t)$ is bounded from above for all $d_0$.
\end{lemma}
\begin{proof}
Since $A(t)\leq   \frac{1}{2}( \frac{\omega_0}{\rho_0} )^2$ and $w>0$, we have
\begin{equation*}
\begin{split}
\dot{d}&= -\frac{1}{2}d^2 + A(t)\rho^2 - \rho +1\\
&\leq  -\frac{1}{2}d^2 +    \frac{1}{2}( \frac{\omega_0}{\rho_0} )^2 \rho^2 - \rho +1\\
&\leq  -\frac{1}{2}d^2  + \max\{ 1, w\rho^2 _M - \rho_M +1 \}.
\end{split}
\end{equation*}
Thus,
$$d(t) \leq \max \big{\{} d_0, \sqrt{2\max\{ 1, w\rho^2 _M - \rho_M +1 \} }  \big{\}}.$$
\end{proof}

The last step of proving the theorem is to combine the comparison principle in Lemma \ref{lemma_comp} with Lemma \ref{lemma_invariant}.
Note that $\Omega$ is  an open set and given any initial data $(\rho_0 , d_0) \in \Omega$ for system \ref{comp_epsystem}, we can find $\epsilon>0$ and initial data $(a_0 , b_0):=(\rho_0 +\epsilon , d_0 -\epsilon) \in \Omega$ for system \ref{comp_auxsystem}. Therefore, by lemmata  \ref{lemma_comp} and \ref{lemma_invariant},
$$0<\rho(t) < a_0, \ \ and \ \ 0< d(t), \ \ \forall t\geq 0.$$
In addition to this, by Lemma \ref{lemma_dbound}, $\rho(t)<a_0$ implies that $d(t)$ is bounded from above for all $t\geq 0$. This completes the proof.

$$$$

\section{Proof of Theorem \ref{thm_exp}}

We start this section by considering the following nonlinear ode system with the time dependent coefficient,
\begin{equation*}
\left\{
  \begin{array}{ll}
    \dot{b} = -b^2 /2 - e^t  a^2 -a +1, \\
    \dot{a} = -ba.\\
  \end{array}
\right.
\end{equation*}

Setting $B(t)=e^t$, one can rewrite the system as follows:
\begin{equation}\label{3by3_system_exp}
\left\{
  \begin{array}{ll}
    \dot{b} = -\frac{1}{2}b^2 - B a^2 -a +1, \\
    \dot{a} = -ba\\
    \dot{B} =B
  \end{array}
\right.
\end{equation}
with $(a , b , B)\big{|}_{t=0} = (a_0 , b_0 , B_0 =1).$

We shall find set of initial data for which the solution of \eqref{3by3_system_exp} exists for all time. Consider surface
$$b=m(B-1) a - n(B-1), \ B \geq 1$$
in $(a,b,B)$ space where $m(\cdot)$ and $n(\cdot)$ are positive on $[0,\infty$) and continuously differentiable. We find conditions on $m(\cdot)$ and $n(\cdot)$ such that trajectory $(a,b,B)$ stays on one side of
$$F(a,b,B)=b-m(B-1)a + n(B-1)=0.$$
In order to do that, it requires
\begin{equation}\label{exp_gradient}
\langle \dot{a}, \dot{b}, \dot{B} \rangle \cdot \nabla F >0,
\end{equation}
on the surface $F(a ,b , B)=0$, where
$$\nabla F = \langle -m(B-1), 1, -m'(B-1)a + n'(B-1) \rangle.$$ 

Expanding the dot product on the surface the left hand side of \eqref{exp_gradient} can be written as
\begin{equation*}
\begin{split}
&\big{\langle} -ba,  -\frac{1}{2}b^2 - B a^2 -a +1 , B \big{\rangle} \cdot \langle -m(B-1), 1, -m'(B-1)a + n'(B-1) \rangle \\
&\Rightarrow ba m(B-1) -\frac{1}{2}b^2 - B a^2 -a +1 -Bm'(B-1)a + Bn'(B-1)\\
&\Rightarrow (ma -n) a m - \frac{1}{2}(ma -n)^2 - B a^2 - a +1 -Bm' a + Bn'\\
&\Rightarrow \bigg{(} \frac{1}{2}m^2 - B \bigg{)} a^2 - (1+Bm')a - \frac{1}{2}n^2 + Bn' +1. 
\end{split}
\end{equation*}
Here and below $m$ and $n$ are evaluated at $B-1$. Thus, on the surface $F(a ,b , B)=0$, \eqref{exp_gradient} is equivalent to
\begin{equation}\label{qd_1_exp}
\bigg{(} \frac{1}{2}m^2 - B \bigg{)} a^2 - (1+Bm')a - \frac{1}{2}n^2 + Bn' +1 >0.
\end{equation}

We will find $m$ and $n$ such that the above inequality holds for some set of $(a, b ,B)$. The inequality is quadratic in $a$, and the nonnegative root of the quadratic equation is given by
$$R_* (B-1) = \frac{(1+Bm') + \sqrt{(1+Bm')^2 + (m^2 - 2B)(n^2 -2Bn' -2)}}{m^2 - 2B},$$
provided that 
\begin{equation}\label{leading_disc_exp}
m^2 - 2B >0 \ \  and \ \ D=(1+Bm')^2 + (m^2 - 2B)(n^2 -2Bn' -2) \geq 0.
\end{equation}
For each fixed $B$,  $\frac{n}{m}$ is the $a$-intercept of  $F(a ,b , B)=0$, so 
\begin{equation}\label{root_ineq_exp}
R_*  (B-1)< \frac{n(B-1)}{m(B-1)}
\end{equation}
implies
\eqref{qd_1_exp}, for all $a \geq \frac{n(B-1)}{m(B-1)}$. 

Expanding, completing square and simplifying \eqref{root_ineq_exp} give
\begin{equation*}
\begin{split}
& m \sqrt{(1+Bm')^2 + (m^2 - 2B)(n^2 -2Bn' -2)} < m^2 n -2B n - m - Bm' m\\
&\Rightarrow 2(m^2 - 2B) \big{\{}  (1+Bn')m^2 - (1+m'B)mn - B n^2 \big{\}} > 0.
\end{split}
\end{equation*}
Since $m^2 > 2B$, it suffices to have
\begin{equation*}
(1+Bn')m^2 - (1+m'B)mn - B n^2 >0.
\end{equation*}
Note that $m$ and $n$ are evaluated at $B-1$, so let $x=B-1$, and writing the  above inequality in terms of $x$ gives,
\begin{equation}\label{qd_2_exp}
\big{\{}1+(x+1)n'(x) \big{\}}m^2 (x) - \big{\{}1+(x+1)m' (x) \big{\}} m(x)n(x) - (x+1)n^2(x) >0.
\end{equation}

We also rewrite the discriminant in terms of $x$ variable,
\begin{equation}\label{}
D=\big{\{} 1+ (x+1)m'(x) \big{\}}^2 + \big{\{} m^2 (x) -2(x+1) \big{\}} \big{\{} n^2 (x) -2(x+1) n'(x) -2 \big{\}}.
\end{equation}

$$$$

\textbf{Construction of} $\mathbf{n(x)}$ \textbf{and} $\mathbf{m(x)}$. We prove the existences of $m(x)$ and $n(x)$. More precisely, we find simple polynomials 
$$n(x):=n_1 (x+n_2)^N, \ \  and \ \ m(x)=m_1 (x+m_2)^M, \ \ M>N$$
that satisfy \eqref{exp_gradient} for all $x\geq 0$.  Unlike with Section \ref{section_poly}, when $N>0$, the discriminant can be positive but \eqref{qd_2_exp} is not satisfied for large $x$. Indeed, writing the highest terms from \eqref{qd_2_exp} gives
$$n_1 m^2 _1 N x^{N+2M} - n_1 m^2 _1 M x^{N+2M},$$
which is not positive because $M>N$. So we shall seek conditions that will lead to $D<0$ for all $x\geq 0$, which in turn implies \eqref{qd_1_exp}. 

First, we see that 
$$m^2 - 2B = m^2 (B-1) -2B = m^2 (x) - 2(x+1)=m^2 _1 (x+m_2)^2 - 2(x+1)$$
is positive for all $x\geq 0$ when $m_1 > \sqrt{2}$ and $m_2 >1$. Thus \eqref{qd_1_exp} hold if $D<0$. 

\begin{lemma}\label{lemma_onsurface_exp}
If $N<0$, $m_2 >1$ and 
$$M > m_2 \sqrt{2},$$ then $D <0$, for all $x \geq 0$.
\end{lemma}

\begin{proof} 
Since $N<0$, we notice that
$$ n^2 (x) -2(x+1) n'(x) = n^2 _1 (x+n_2)^{2N} - 2(x+1)n_1 N (x+n_2)^{N-1} >0.$$
Thus, it suffices to have
$$\big{\{} 1+ (x+1)m_1 M (x+m_2)^{M-1}  \big{\}}^2 + \big{\{} m^2 _1 (x+m_2)^{2M} -2(x+1)  \big{\}} \cdot (-2) >0.$$
Dropping 1 and the last positive term, we see that the above inequality is satisfied if
$$m^2 _1 (x + m_2)^{2M-2} \big{\{} M^2 (x+1)^2 -2 (x+m_2)^2  \big{\}} >0$$
or
$$\frac{M}{\sqrt{2}} > \frac{ (x+ m_2)}{x+1} = 1+ \frac{m_2 -1 }{x+1}.$$
Since the right hand side is decreasing in $x$, $\frac{M}{\sqrt{2}} > m_2$ gives the desired result.
\end{proof}

From Lemma \ref{lemma_onsurface_exp},  we can see that any trajectory on $F(a,b,B)=0$ can't cross $F(a,b,B)=0$ from left to right. We continue to construct an invariant region for the system \eqref{3by3_system_exp}. It is easy to see that trajectory  $(a, 0, B)$ on the $aB$ plane moves upward if  $a< a^* $ (downward if $a > a^*$) where
$$a^* (B-1)=\frac{1-\sqrt{1+4B}}{-2B}.$$
In order to secure the invariant region we need 
$$a^* (B-1) > \frac{n(B-1)}{m(B-1)},$$
for all $B \geq 1$.  Let $x=B-1$ and the above inequality is fulfilled by the following lemma.

\begin{lemma}\label{lemma_ontheaxis_2}
If $M>N+1$, $m_2 > n_2 >1$ and, 
$$n_2 >  \bigg{\{} \frac{n_1}{2m_1}(1+\sqrt{5}) \bigg{\}}^{1/(M-N-1)},$$
then it holds,
$$a^* (x) > \frac{n(x)}{m(x)},$$
for all $x\geq 0$.
\end{lemma}
\begin{proof}
Since $a^* (x)= \frac{-1 + \sqrt{1 + 4(x+1)}}{2(x+1)}$,  we show that
$$\underbrace{\frac{-1 + \sqrt{1 + 4(x+1)}}{2(x+1)}}_{LHS} > \underbrace{ \frac{n_1 (x+n_2)^N}{m_1 (x+m_2)^M}}_{RHS}.$$
Since $1 < n_2 < m_2$, we have  
$$RHS= \frac{n_1}{m_1} \frac{1}{(x+n_2)^{-N}} \frac{1}{(x+m_2)^M} < \frac{n_1}{m_1} \frac{1}{(x+n_2)^{M-N}}.$$
Thus, it suffice to prove
$$LHS=  \frac{-1 + \sqrt{1 + 4(x+1)}}{2(x+1)} > \frac{n_1}{m_1 (x+ n_2)^{M-N}},$$
or equivalently
$$-1 + \sqrt{1+ 4(x+1)} > \frac{2n_1(x+1)}{m_1(x+n_2)^{M-N}}.$$

Since $n_2 > 1$, the above inequality holds if
$$-1 + \sqrt{1+ 4(x+1)} >  \frac{2n_1 }{ m_1 (x+n _2)^{M-N-1}}. $$
Upon simplification, the inequality is reduced to
\begin{equation}
4(x+1) > \frac{4n^2 _1}{m^2 _1 (x+n_2)^{2(M-N-1)}} + \frac{4n_1}{m_1 (x+n_2)^{M-N-1}}.
\end{equation}
Since $M-N-1 >0$, we note that the right hand side of the above inequality is decreasing in $x$ and the left hand side is increasing in $x$. Thus, it suffices to hold at $x=0$, that is,
$$4 > \frac{4n^2 _1}{m^2 _1}\cdot \frac{1}{ n^{2(M-N-1)} _2} + \frac{4n_1}{m_1}\cdot \frac{1}{n^{M-N-1} _2},$$
or equivalently 
$$m^2 _1 n^{2(M-N-1)} _2 - m_1 n_1 n^{M-N-1} _2 - n^2 _1 >0.$$
The left hand side is quadratic in $n^{M-N-1} _2$ and positive when $n^{M-N-1} _2$ is greater than the positive root of the quadratic equation. That is,
$$n^{M-N-s} _2 > \frac{m_1 n_1 + \sqrt{m^2 _1 n^2 _1 + 4 m^2 _1 n^2 _1}}{2m^2 _1}= \frac{n_1}{2m_1}(1+\sqrt{5}).$$
\end{proof}

We are done with constructing the ``shrinking" invariant region and we recycle  Lemma  \ref{lemma_invariant} to summarize its property. Furthermore, the monotone properties in Lemma \ref{lemma_comp} hold between
\begin{equation}
\left\{
  \begin{array}{ll}
    \dot{d} = -d^2/2 +A(t) \rho^2 -\rho +1, \\
    \dot{\rho} = -d\rho\\
  \end{array}
\right.
\end{equation}
and
\begin{equation}
\left\{
  \begin{array}{ll}
    \dot{b} = -b^2 /2 -e^t a^2 -a +1, \\
    \dot{a} = -ba.\\
  \end{array}
\right.
\end{equation}
The remaining part of the proof is analogous to the proof of Theorem \ref{thm_poly}.

\bibliographystyle{abbrv}

\end{document}